\documentclass[12pt,draft,reqno]{amsart}
\usepackage{amsmath}
\usepackage{amssymb}
\usepackage{amsthm}
\usepackage{bm}
\numberwithin{equation}{section}
\newtheorem{theorem}{Theorem}[section]
\newtheorem{lemma}[theorem]{Lemma}
\newtheorem{proposition}[theorem]{Proposition}

\newtheorem{definition}[theorem]{Definition}

\theoremstyle{definition}

\allowdisplaybreaks
\begin{document}
\title[Real zeros of Hurwitz-Lerch zeta functions in the interval $(-1,0)$]{Real zeros of Hurwitz-Lerch zeta functions \\in the interval ${\bm{(-1,0)}}$}
\author[T.~Nakamura]{Takashi Nakamura}
\address[T.~Nakamura]{Department of Liberal Arts, Faculty of Science and Technology, Tokyo University of Science, 2641 Yamazaki, Noda-shi, Chiba-ken, 278-8510, Japan}
\email{nakamuratakashi@rs.tus.ac.jp}
\urladdr{https://sites.google.com/site/takashinakamurazeta/}
\subjclass[2010]{Primary 11M35}
\keywords{Real zeros of Hurwitz-Lerch zeta function}
\maketitle

\begin{abstract}
For $0 < a \le 1$, $s,z \in {\mathbb{C}}$ and $0 < |z|\le 1$, the Hurwitz-Lerch zeta function is defined by $\Phi (s,a,z) := \sum_{n=0}^\infty z^n(n+a)^{-s}$ when $\sigma :=\Re (s) >1$.
In this paper, we show that $\Phi (\sigma,a,z) \ne 0$ when $\sigma \in (-1,0)$ if and only if [I] $z=1$ and $(3-\sqrt{3}) /6 \le a  \le 1/2$ or $(3+\sqrt{3}) /6 \le a  \le 1$, [II] $z \in [-1,1)$ and $(1-z)(1-a) \le 1$, [III] $z \not \in {\mathbb{R}}$ and $0<a \le 1$. In addition, we give a new proof of the functional equation of $\Phi (s,a,z)$. 
\end{abstract}

\section{Introduction and statement of main results}

\subsection{Main results}
The Hurwitz-Lerch zeta function is defined by as follows.
\begin{definition}[see {\cite[p.~53, (1)]{Er}}]\label{def:ler}
Let $0 < a \le 1$, $s,z \in {\mathbb{C}}$ and $0< |z|\le 1$. Then the Hurwitz-Lerch zeta function $\Phi(s,a,z)$ is defined by
\begin{equation}
\Phi (s,a,z) := \sum_{n=0}^{\infty} \frac{z^n}{(n+a)^s}, \qquad s := \sigma + it, \quad \sigma >1 , \quad t \in {\mathbb{R}}.
\end{equation}
\end{definition}
We can easily see that the Riemann zeta function $\zeta (s)$ and the Hurwitz zeta function $\zeta (s,a)$ are expressed as $\Phi (s,1,1)$ and $\Phi (s,a,1)$, respectively. The Dirichlet series of $\Phi(s,a,z)$ converges absolutely in the half-plane $\sigma >1$ and uniformly in each compact subset of this half-plane. When $z \ne 1$, the function $\Phi(s,a,z)$ is analytically continuable to the whole complex plane. However, the Hurwitz zeta function $\zeta (s,a)$ is a meromorphic function with a simple pole at $s=1$. In this paper, we show the following.
\begin{theorem}\label{th:hlz1}
Let $b_2^\pm := (3\pm \sqrt{3}) /6$. Then the Hurwitz-Lerch zeta function $\Phi (\sigma,a,z)$ does not vanish when $\sigma \in (-1,0)$ if and only if
$$
\left\{\begin{matrix} 
{\rm{[I]}} &  z=1, &  b_2^-\le a  \le 1/2 \,\,\, {\rm{or}} \,\,\, b_2^+ \le a  \le 1,\\
{\rm{[II]}} &  z \in [-1,1), & (1-z)(1-a) \le 1,\\
{\rm{[III]}} & \, |z|=1, \,\, z \not \in {\mathbb{R}}, & 0<a \le 1.
\end{matrix} \right.
$$
Note that $b_2^\pm$ are the roots of the second Bernoulli polynomial $B_2 (x) := x^2-x+1/6$.
\end{theorem}
In Section 2, we prove Theorem \ref{th:hlz1}. During the proof process of Theorem \ref{th:hlz1} (see Proposition \ref{pro:defi1}), we show $\zeta (\sigma,a) > 0$ when $b_2^- \le a \le 1/2$ and $\zeta (\sigma,a) < 0$ when $b_2^+ \le a \le 1$ in the interval $(-1,0)$. In Section 3, we give a new proof of the functional equations (\ref{eq:hueq1}) and (\ref{eq:hueq2}) by using the integral representations (\ref{eq:irhur1}) and (\ref{eq:hlzinrp2}), and modifying the fifth method of the proof of the functional equation for the Riemann zeta function (see \cite[Section 2.8]{Tit}). It should be noted that in the 21st century, Knopp and Robins \cite{KR}, and Navas, Ruiz and Varona \cite{NRV} gave new proofs of the functional equation for $\zeta (s,a)$ by using Poisson summation of the Lipschitz summation formula (see \cite[p.~1916]{KR}) and the uniqueness of Fourier coefficients (see \cite[p.~190]{NRV}), respectively.

\subsection{Known results}
As mentioned in \cite[Section 1]{NaC}, the study of real zeros of $\Phi (s,a,z)$ has a long, long history. Let ${\rm{Li}}_s (z) := z\Phi (s,1,z) = \sum_{n=1}^\infty z^n n^{-s}$. Roy \cite{Roy} proved that ${\rm{Li}}_\sigma (z) \ne 0$ for all $|z| \le 1$, $z \ne 1$ and $\sigma >0$. Berndt \cite{Ber} showed that $\zeta (\sigma,a+1) = \zeta (\sigma,a) - a^{-\sigma} \ne 0$ for any $0 < \sigma <1$ and $0 \le a \le 1$. 

For zeros of zeta functions in the half-plane $\sigma <0$, we have the following research. Peyerimhoff \cite{Pey} proved that for (fixed) $\sigma<0$, the function ${\rm{Li}}_\sigma (z$) has $-\lfloor \sigma \rfloor$ simple zeros for $z \in {\mathbb{C}} \setminus [1,\infty)$ and they are all non-positive (see also \cite[Section 8]{OS}). Spira \cite{Spira} showed that if $\sigma \le -4a-1-2[1-2a]$ and $|t|\le1$, then $\zeta (s,a) \ne 0$ except for zeros on the negative real line, one in each interval $(-2n-4a-1,-2n-4a+1)$, ${\mathbb{N}} \ni n \ge 1-2a$. Some similar results for the Lerch zeta function $\Phi (s,a,e^{2 \pi i \theta})$ with $0 < \theta \le 1$ are given by Garunk\v{s}tis and Laurin\v{c}ikas \cite{GauLa}. Veselov and Ward \cite{VW} proved that $\zeta(-\sigma,a)$ has no real zeros in the region $4\pi e a > 1+2\sigma + \log \sigma$ for large $\sigma$. 

For real zeros of the Hurwitz zeta function in the interval $(0,1)$, Schipani \cite{Schi} showed that $\zeta (\sigma,a)$ has no zeros when $0<\sigma <1$ and $1-\sigma \le a$. The author \cite{NaC} proved that $\zeta (\sigma,a)$ does not vanish for all $0 <\sigma <1$ if and only if $a \ge 1/2$, and $\Phi (\sigma,a,z) \ne 0$ for all $0 <\sigma <1$ and $0 < a \le 1$ when $z \ne 1$. 

\subsection{Some remarks}
The Dirichlet {\textit{L}}-function with a Dirichlet character $\chi$ is defined by $L(s,\chi) := \sum_{n=1}^\infty \chi (n) n^{-s}$. Let $\varphi$ be the Euler totient function and $\chi$ be a primitive Dirichlet character of conductor of $q$. Then the following six relations between $L(s,\chi)$, $\zeta (s,a)$ and ${\rm{Li}}_s  (z)$ are well-known;
\begin{equation*}
\begin{split}
&L (s,\chi) = q^{-s} \sum_{r=1}^q \chi (r) \zeta (s,r/q) ,\qquad
\zeta (s,r/q) = 
\frac{q^s}{\varphi (q)} \sum_{\chi \!\!\! \mod q} \overline{\chi} (r) L(s,\chi) , \\
&\zeta (s,r/q) = \sum_{n=1}^q e^{-2\pi i rn/q} {\rm{Li}}_s  (e^{2\pi i rn/q}), \qquad
{\rm{Li}}_s  (e^{2\pi i r/q}) = q^{-s} \sum_{n=1}^q e^{2\pi irn/q} \zeta (s,n/q), \\
&L (s,\chi) = 
\frac{1}{G(\overline{\chi})} \sum_{r=1}^q \overline{\chi} (r) {\rm{Li}}_s  (e^{2\pi i r/q}) ,\qquad
{\rm{Li}}_s  (e^{2\pi i r/q}) =
\frac{1}{\varphi (q)} \sum_{\chi \!\!\! \mod q} G(\overline{\chi}) L(s,\chi) ,
\end{split}
\end{equation*}
where $G(\overline{\chi}) := \sum_{n=1}^q \overline{\chi}(n)e^{2\pi irn/q}$ denotes the Gauss sum associated to $\overline{\chi}$. 

Siegel \cite{Sie} proved that for any $\varepsilon >0$, there exists $C_\varepsilon>0$ such that, if $\chi$ is a real primitive Dirichlet character modulo $q$, then $L(1,\chi) > C_\varepsilon q^{-\varepsilon}$. It is considered likely that $L(\sigma,\chi) \ne 0$ for all $0 <\sigma <1$, namely, so-called Siegel zeros do not exist. From the Euler product and the functional equation of the Dirichlet {\textit{L}}-function, $L(\sigma,\chi)$ does not vanish for $\sigma \in (-1,0)$. This fact should be compared with the following deduced by Theorem \ref{th:hlz1} that $\zeta (\sigma,a) \ne 0$ for all $\sigma \in (-1,0)$ if and only if $(3-\sqrt{3}) /6 \le a  \le 1/2$ or $(3+\sqrt{3}) /6 \le a  \le 1$ and ${\rm{Li}}_\sigma (z) \ne 0$ for all $z \ne 1$, $|z| =1$ and $\sigma \in (-1,0)$. 

Taking known results and Theorem \ref{th:hlz1} altogether, we have the following.
$$
\begin{matrix} 
{} & -1 < \sigma <0 & 0<\sigma<1 & 1 <\sigma \\
\zeta (\sigma, a) \ne 0 & b_2^-\le a  \le 1/2, \,\, b_2^+ \le a \le 1 & a \ge 1/2 & 0 <a \le 1\\
{\rm{Li}}_\sigma (z) \ne 0, \,\, z \ne 1 &  z \ne 1 & z \ne 1 & z \ne 1 \\
L(\sigma,\chi)\ne 0, \mbox{ real $\chi$} & \quad \mbox{real } \chi \quad & \mbox{ Siegel zero } & \mbox{real } \chi
\end{matrix} 
$$

\section{Proof of Theorem \ref{th:hlz1}}
\subsection{Proof of Theorem \ref{th:hlz1} (I)}
To prove (I) of Theorem \ref{th:hlz1}, we define $H(a,x)$ by
\begin{equation}
\label{eq:defHax}
H(a,x) := \frac{e^{(1-a)x}}{e^x-1} - \frac{1}{x} =  \frac{xe^{(1-a)x} - e^x +1}{x(e^x-1)}, \qquad x>0.
\end{equation}
We have to mention that for $|x|\le 1$,
\begin{equation}
\label{eq:haxtay}
H(a,x) = \frac{1}{x} \sum_{n=1}^\infty B_n(1\!-\!a) \frac{x^n}{n!} = \frac{1}{2} -a + \frac{B_2(a)}{2!}x+\cdots,
\end{equation}
where $B_n (a)$ is the $n$-th Bernoulli polynomial (see for example \cite[p.~246]{Apo} or \cite[Proposition 4.9]{AIK}). We quote the following lemma from \cite{NaC}.
\begin{lemma}[{see \cite[Lemma 2.1]{NaC}}]
\label{lem:12.2ac}
For $0 < \sigma <1$, 
\begin{equation}
\label{eq:gamhurac}
\Gamma (s) \zeta (s,a) =  \int_0^\infty \biggl( \frac{e^{(1-a)x}}{e^x-1} - \frac{1}{x} \biggr) x^{s-1} dx.
\end{equation}
\end{lemma}

When $-1 < \sigma < 0$, we have the following integral representation.
\begin{proposition}[{see \cite[(2.4)]{NRV}}]\label{pro:irhur1}
For $-1 < \sigma <0$, it holds that
\begin{equation}\label{eq:irhur1}
\Gamma (s) \zeta (s,a) = \int_0^\infty 
\biggl( \frac{e^{(1-a)x}}{e^x-1} - \frac{1}{x} - \frac{1}{2}+a \biggr) x^{s-1} dx .
\end{equation}
\end{proposition}

\begin{proof}
For reader's convenience, we give a proof here. When $0< \sigma < 1$, we have
$$
\Gamma (s) \zeta (s,a) = \int_0^\infty \!\! H(a,x) x^{s-1} dx
$$
by (\ref{eq:defHax}) and Lemma \ref{lem:12.2ac}. Hence we have
\begin{equation*}
\begin{split}
&\Gamma (s) \zeta (s,a) = 
\int_0^1 \!\! H(a,x) x^{s-1} dx + \int_1^\infty \!\! H(a,x) x^{s-1} dx \\ = &
\int_0^1 \Bigl( H(a,x) - \frac{1}{2}+a \Bigr) x^{s-1} dx + 
\Bigl( \frac{1}{2}-a \Bigr) \int_0^1 x^{s-1} dx + \int_1^\infty \!\! H (a,x) x^{s-1} dx \\ = &
\int_0^1 G(a,x) x^{s-1} dx +
\int_1^\infty \!\! H (a,x) x^{s-1} dx + \Bigl( \frac{1}{2}-a \Bigr) \frac{1}{s} ,
\end{split}
\end{equation*}
where the function $G (a,x)$ is defined as
\begin{equation}\label{eq:defg1}
G (a,x) :=  \frac{e^{(1-a)x}}{e^x-1} - \frac{1}{x} - \frac{1}{2}+a = H(a,x) - \Bigl( \frac{1}{2}-a \Bigr).
\end{equation}

For $\sigma > -1$, it holds that
$$
\int_0^1 \bigl| G(a,x) \bigr| \bigl| x^{s-1} \bigr| dx \ll \int_0^1 x^{\sigma} dx = \frac{1}{1+\sigma} < \infty 
$$
by (\ref{eq:haxtay}) and (\ref{eq:defg1}).
On the other hand, we have
\begin{equation}\label{eq:s-1}
\frac{1}{s} = - \int_1^\infty \!\! x^{s-1} dx, \qquad -1 < \sigma <0. 
\end{equation}
Furthermore, when $-1 < \sigma < 0$ one has
$$
\int_1^\infty \bigl| H(a,x) \bigr| \bigl| x^{s-1} \bigr| dx \ll
\int_1^\infty \frac{e^{(1-a)x}}{e^x-1}x^{\sigma -1} dx < \infty , \qquad 
\int_1^\infty \bigl| x^{s-1} \bigr| dx = \int_1^\infty \!\! x^{\sigma-1} dx  < \infty. 
$$
Therefore, the integral representation 
$$
\Gamma (s) \zeta (s,a) = \int_0^1 G(a,x) x^{s-1} dx +
\int_1^\infty \!\! H (a,x) x^{s-1} dx - \Bigl( \frac{1}{2}-a \Bigr) \int_1^\infty \!\! x^{s-1} dx
$$
gives an analytic continuation for $-1<\sigma <0$. Hence we obtain this Proposition. 
\end{proof}

\begin{lemma}\label{lem:negdefi}
The function $G(a,x)$ defined by (\ref{eq:defg1}) is negative for all $x>0$ if and only if $b_2^- :=(3-\sqrt{3})/6 \le a \le 1/2$. Moreover, we have $G(a,x)>0$ for all $x>0$ if and only if $a \ge (3+\sqrt{3})/6 =: b_2^+$. 
\end{lemma}
\begin{proof}
Obviously, one has
$$
G (a,x) = \frac{xe^{(1-a)x} -e^x+1 - (1/2-a)x(e^x-1)}{x(e^x-1)} .
$$
Since $x(e^x-1) >0$ for all $x>0$, we consider the numerator of $G (a,x)$ written as
$$
g(a,x) := {x(e^x-1)} G (a,x)=  xe^{(1-a)x} -e^x+1 - (1/2-a)x(e^x-1).
$$
By differentiating with respect to $x$, we have
\begin{equation*}
\begin{split}
& g' (a,x) = (1-a) xe^{(1-a)x} + e^{(1-a)x} - e^x- (1/2-a) (xe^x +e^x-1), \\
& g'' (a,x) = \bigl( (1-a)^2x +2(1-a) \bigr) e^{(1-a)x} - e^{x} - (1/2-a) (xe^x +2e^x).
\end{split}
\end{equation*}
Obviously, one has 
\begin{equation}\label{eq:ggg0}
g(a,0) = g'(a,0)= g''(a,0)=0. 
\end{equation}
Now we consider the function
$$
e^{(a-1)x} g'' (a,x) = (1-a)^2x +2(1-a) - (1/2-a) xe^{ax} - 2(1-a) e^{ax} .
$$

First suppose $0< a  \le 1/2$. Then we have $\lim_{x \to \infty}e^{(a-1)x} g'' (a,x) = -\infty$. By the Taylor expansion of $e^{ax}= \sum_{n=0}^\infty (a x)^n/n!$, one has
\begin{equation*}
\begin{split}
&(1/2-a) xe^{ax}  + 2(1-a) e^{ax} = 
(1/2-a) \bigl( x + ax^2+\cdots \bigr) +2 (1-a) \bigl( 1 + ax+\cdots \bigr) \\ &=
2(1-a) + \bigl( (1/2-a) + 2a(1-a) \bigl) x + \cdots .
\end{split}
\end{equation*}
In this case, we have
$$
(1-a)^2x +2(1-a) <  (1/2-a) xe^{ax}  +2(1-a) e^{ax}, \qquad x>0
$$
if $(1-a)^2 \le (1/2-a) + 2a(1-a)$ which is equivalent to $3B_2(a) =3a^2 -3a+1/2 \le 0$, where $B_2(a)$ is the second Bernoulli polynomial. Hence one has $g'' (a,x) <0$ for all $x>0$ when $(3-\sqrt{3}) /6 \le a  \le 1/2$. Thus we obtain $g' (a,x) <0$ and $g (a,x) <0$ for all $x>0$ when $b_2^- \le a  \le 1/2$ from (\ref{eq:ggg0}). 

Next assume $a \ge 1/2$. Then one has $\lim_{x \to \infty}e^{(a-1)x} g'' (a,x) = \infty$. By the Taylor expansion of $e^{ax}= \sum_{n=0}^\infty (a x)^n/n!$, we have
$$
(a-1/2) xe^{ax} = \sum_{n=1}^\infty \frac{(a-1/2)a^{n-1}}{(n-1)!} x^n ,\qquad
2(1-a) e^{ax} = \sum_{n=0}^\infty \frac{2(1-a) a^n}{n!} x^n .
$$
For $n \ge 2$ and $a \ge 3/4$, it holds that
$$
 \frac{(a-1/2)a^{n-1}}{(n-1)!} \ge \frac{2(1-a) a^n}{n!}
$$
since we have
$$
n(a-1/2) \ge 2 (a-1/2) \ge 1/2 \ge -2(a-1/2)^2 + 1/2 = 2a(1-a). 
$$
Note that $(3+\sqrt{3})/6 = 0.788675... > 3/4$. Thus we have
$$
(1-a)^2x +2(1-a) + (a-1/2) xe^{ax} > 2(1-a) e^{ax} , \qquad x>0
$$
if $(1-a)^2 + (a-1/2) \ge 2a(1-a)$ which is equivalent to $3B_2(a) =3a^2 -3a+1/2 \ge 0$. Thus we obtain $g'' (a,x) >0$ for all $x>0$ when $a \ge (3+\sqrt{3})$. Hence one has $g' (a,x) >0$ and $g (a,x) >0$ for all $x>0$ when $a \ge b_2^+$ from (\ref{eq:ggg0}). 

Finally suppose $0 < a < b_2^-$ or $1/2 < a < b_2^+$. By (\ref{eq:haxtay}) and (\ref{eq:defg1}), one has
$$
2G(a,x) = B_2(a)x + O(x^2).
$$
Hence we have
$$
G(a,x) > 0, \quad 0 < a < b_2^- \quad \mbox{ and } \quad G(a,x) < 0, \quad 1/2 < a < b_2^+ 
$$
when $x>0$ is sufficiently small. On the other hand, it holds that
$$
G(a,x) < 0, \quad 0 < a < b_2^- \quad \mbox{ and } \quad G(a,x) > 0, \quad 1/2 < a < b_2^+ 
$$
when $x>0$ is sufficiently large from the term $(1/2-a)x(e^x-1)$ in $G(a,x)$. Hence we have this lemma.
\end{proof}

By using Proposition \ref{pro:irhur1}, Lemma \ref{lem:negdefi} and the fact that $\Gamma (\sigma)<0$ when $-1<\sigma <0$, we have the following.
\begin{proposition}\label{pro:defi1}
For $-1<\sigma <0$, one has 
$$
\zeta (\sigma,a) > 0, \quad b_2^- \le a \le 1/2 \quad \mbox{ and } \quad 
\zeta (\sigma,a) < 0, \quad b_2^+ \le a \le 1. 
$$
\end{proposition}
\begin{proof}[Proof of Theorem \ref{th:hlz1} (I)]
We only have to consider the cases $0 < a < b_2^-$ or $1/2 < a < b_2^+$ from Proposition \ref{pro:defi1}. It is well-known that
$$
\zeta (0,a) = \frac{1}{2} -a , \qquad 
\zeta (-1,a) = -\frac{a^2}{2} + \frac{a}{2} - \frac{1}{12} = - \frac{B_2(a)}{2}
$$
(see \cite[Proposition 9.3]{AIK} or \cite[Theorem 12.3]{Apo}). Thus it holds that
\begin{equation*}
\begin{split}
&\zeta (0,a) > 0 \quad \mbox{and} \quad \zeta (-1,a) < 0, \qquad \mbox{when} \quad 0 < a < b_2^-, \\
&\zeta (0,a) < 0 \quad \mbox{and} \quad \zeta (-1,a) > 0, \qquad \mbox{when} \quad 1/2 < a < b_2^+.
\end{split}
\end{equation*}
Therefore, we have Theorem \ref{th:hlz1} (I).
\end{proof}

\subsection{Proof of Theorem \ref{th:hlz1} (II) and (III)}
We quote the following integral representation of Hurwitz-Lerch zeta function $\Phi (s,a,z)$.
\begin{lemma}[{see \cite[p.~53, (3)]{Er}}]
When $z \ne 1$, we have
\begin{equation}\label{eq:hlzinrp1}
\Phi (s,a,z) = \frac{1}{\Gamma (s)} \int_0^\infty \frac{x^{s-1} e^{(1-a)x}}{e^x-z} dx, \qquad \Re (s) >0. 
\end{equation}
\end{lemma}
By using the integral formula above, we obtain the following.
\begin{proposition}
For $z \ne 1$, one has
\begin{equation}\label{eq:hlzinrp2}
\Phi (s,a,z) = \frac{1}{\Gamma (s)} \int_0^\infty 
\biggl( \frac{e^{(1-a)x}}{e^x-z} -\frac{1}{1-z} \biggr)x^{s-1} dx, 
\qquad -1 < \Re (s) < 0. 
\end{equation}
\end{proposition}
\begin{proof}
When $\sigma >0$, we have
\begin{equation*}
\begin{split}
&\Gamma (s) \Phi (s,a,z) = 
\int_0^1 \frac{x^{s-1} e^{(1-a)x}}{e^x-z} dx + \int_1^\infty \frac{x^{s-1} e^{(1-a)x}}{e^x-z} dx \\ =&
\int_0^1 \biggl( \frac{e^{(1-a)x}}{e^x-z} -\frac{1}{1-z} \biggr)x^{s-1} dx
+ \int_0^1 \frac{x^{s-1}}{1-z} dx + \int_1^\infty \frac{x^{s-1} e^{(1-a)x}}{e^x-z} dx \\ =&
\int_0^1 \biggl( \frac{e^{(1-a)x}}{e^x-z} -\frac{1}{1-z} \biggr)x^{s-1} dx 
+ \int_1^\infty \frac{x^{s-1} e^{(1-a)x}}{e^x-z} dx + \frac{1}{s(1-z)} 
\end{split}
\end{equation*}
by using (\ref{eq:hlzinrp1}). From
$$
\lim_{x \to 0+} \biggl( \frac{e^{(1-a)x}}{e^x-z} -\frac{1}{1-z} \biggr) =0,
$$
we can see that
$$
\int_0^1 \biggl| \frac{e^{(1-a)x}}{e^x-z} -\frac{1}{1-z} \biggr| \bigl|x^{s-1} \bigr| dx \ll
\int_0^1 x^{\sigma} dx < \infty
$$
when $-1 < \Re (s) < 0$. Hence we have this proposition by (\ref{eq:s-1}). 
\end{proof}

In order to prove Theorem \ref{th:hlz1} (II), we show the following.
\begin{proposition}
When $z \in [-1,1)$ and $(1-a)(1-z) \le 1$, we have
$$
\Phi (\sigma,a,z) >0, \qquad -1 < \sigma <0.
$$
\end{proposition}
\begin{proof}
Let $z \in [-1,1)$ and put
$$
G_{z} (a,x) := \frac{e^{(1-a)x}}{e^x-z} - \frac{1}{1-z} = \frac{(1-z)e^{(1-a)x} - e^x+z}{(1-z)(e^x-z)}.
$$
Hence consider the numerator of $G_{z} (a,x)$ expressed as
$$
g_{z} (a,x) := (1-z)e^{(1-a)x} - e^x+z .
$$
Then we obtain $g_{z} (a,0)=0$ and
$$
g_{z} ' (a,x) = (1-z)(1-a)e^{(1-a)x} - e^x.
$$
Obviously $g_{z} ' (a,x) <0$ is equivalent to $(1-z)(1-a) < e^{ax}$. We can see that $(1-z)(1-a) < e^{ax}$ for all $x>0$ if  $(1-a)(1-z) \le 1$. Hence, we have this proposition by (\ref{eq:hlzinrp2}) and the fact that $\Gamma (\sigma) <0$ when $-1 < \sigma <0$.
\end{proof}

\begin{proof}[Proof of Theorem \ref{th:hlz1} (II)]
We only have to prove that $\Phi (\sigma,a,z)$ has zeros in the interval $(-1,0)$ when  $(1-a)(1-z) > 1$. From \cite[p.~17]{KKY}, one has
$$
\Phi (0,a,z) = \frac{1}{1-z} >0, \qquad \Phi (-1,a,z) = \frac{a}{1-z} + \frac{z}{(1-z)^2}.
$$
When $(1-a)(1-z) > 1$, we have $\Phi (-1,a,z) < 0$. Therefore, $\Phi (\sigma,a,z)$ vanishes in the interval $(-1,0)$ if $(1-a)(1-z) > 1$ by the intermediate value theorem. 
\end{proof}

\begin{proof}[Proof of Theorem \ref{th:hlz1} (III)]
For $0<r\le 1$, $0 < \theta < \pi$ and $\pi < \theta < 2\pi$, put
\begin{equation*}
\begin{split}
&G_{r,\theta} (a,x) := \frac{e^{(1-a)x}}{e^x-re^{i \theta}} - \frac{1}{1-re^{i \theta}} \\&=
\frac{e^{(1-a)x}((e^x-r\cos \theta)+ir\sin \theta)}{(e^x-r\cos \theta)^2+r^2\sin \theta}
- \frac{(1-r\cos \theta)+ir\sin \theta}{(1-r\cos \theta)^2+r^2\sin \theta}.
\end{split}
\end{equation*}
Obviously, we obtain that
$$
\Im \bigl( G_{r,\theta} (a,x) \bigr) = 
\frac{e^{(1-a)x} r\sin \theta}{e^{2x} +r^2 - 2e^xr\cos \theta}- \frac{r\sin \theta}{1+r^2-2r\cos \theta} .
$$
Now we consider the following three functions
\begin{equation*}
\begin{split}
& g_{r,\theta}^\flat (a,x) := e^{(1-a)x} \bigl( 1+r^2-2r\cos \theta \bigr) , \qquad 
g_{r,\theta}^\sharp (a,x) := e^{2x} +r^2 - 2e^xr\cos \theta, \\
& g_{r,\theta}^\natural(a,x) := e^{2(1-a)x} +r^2 - 2e^{(1-a)x}r\cos \theta .
\end{split}
\end{equation*}
For all $x>0$, we can show $g_{r,\theta}^\flat (a,x) < g_{r,\theta}^\natural (a,x) < g_{r,\theta}^\sharp (a,x)$ from 
\begin{equation*}
\begin{split}
& g_{r,\theta}^\sharp (a,x) - g_{r,\theta}^\natural(a,x) =
\bigl( e^x - e^{(1-a)x} \bigr) \bigl( e^x + e^{(1-a)x} -2r\cos \theta \bigr) >0, \\
& g_{r,\theta}^\natural(a,x) - g_{r,\theta}^\flat (a,x) = 
\bigl( e^{(1-a)x} - 1 \bigr) \bigl( e^{(1-a)x} - r^2 \bigr) > 0.
\end{split}
\end{equation*}
Hence $\Im (G_{r,\theta} (a,x) / \sin \theta) <0$ for all $x>0$. Therefore, we obtain $\Im (\Phi (\sigma,a,z)) \ne 0$ when $-1 < \Re (s) < 0$ for $z=re^{i\theta}$, where $0<r\le 1$, $0 < \theta < \pi$ and $\pi < \theta < 2\pi$.
\end{proof}

\section{New proofs of the functional equations}
\subsection{The case ${\bm{z=1}}$}
In this subsection, we show the functional equation
\begin{equation}\label{eq:hueq1}
\zeta (s,a) = \frac{(-\pi i) (2\pi)^{s-1}}{\Gamma(s) \sin \pi s}
\Biggl( e^{\pi is/2} \sum_{n=1}^\infty \frac{e^{2\pi ina}}{n^{1-s}} -
e^{-\pi is/2} \sum_{n=1}^\infty \frac{e^{-2\pi ina}}{n^{1-s}} \Biggr), 
\end{equation}
where $0<a<1$ and $\Re (s) <0$ (see for instance \cite[Theorem 9.4]{AIK}) by the integral representation (\ref{eq:irhur1}). The following formulas are well-known.
\begin{lemma}[{see the proof of \cite[Theorem 4.11]{AIK}}]
For $0<a<1$, one has
$$
\frac{e^{ax}}{e^x-1} - \frac{1}{x} = \lim_{N \to \infty} \sum_{0 \ne n =N}^{-N} \frac{e^{2\pi ina}}{x-2\pi in},
\qquad 
a-\frac{1}{2} = \lim_{N \to \infty} \sum_{0 \ne n =N}^{-N} \frac{e^{2\pi i na}}{-2\pi in}.
$$
\end{lemma}
Recall $G(a,x)$ is defined by (\ref{eq:defg1}). From the lemma above, we have 
\begin{equation}\label{eq:ft}
\begin{split}
G(a,x) &= \lim_{N \to \infty} \sum_{0 \ne n =N}^{-N} \biggl( \frac{e^{-2\pi i na}}{x-2\pi in} 
- \frac{e^{2\pi i na}}{2\pi in}\biggr) \\ &=
\sum_{n=1}^\infty \frac{x e^{-2\pi i na}}{2\pi i n(x-2\pi i n)} - 
\sum_{n=1}^\infty \frac{x e^{2\pi i na}}{2\pi i n(x+2\pi i n)} .
\end{split}
\end{equation}

\begin{lemma}\label{lem:mt}
When $-1 < \Re (s) < 0$, we have
\begin{equation*}
\begin{split}
&\int_0^\infty \!\! \frac{x^s dx}{2\pi i n(x-2\pi i n)} = 
\frac{(2\pi)^s e^{\pi is/2}}{1-e^{2\pi is}} n^{s-1}, \qquad
\int_0^\infty \!\! \frac{x^s dx}{2\pi i n(x+2\pi i n)} = 
\frac{(2\pi)^s e^{3\pi is/2}}{1-e^{2\pi is}} n^{s-1}.
\end{split}
\end{equation*}
\end{lemma}
\begin{proof}
Let $R>r >0$. Then consider the following contour integral, \\
$\,\,\,\,\; C_1:$ the part of real axis from $r$ to $R$, \\
$I(R):$ the circle of radius $R$ with center at the origin (oriented counter-clockwise), \\
$\,\,\,\,\; C_2:$ the part of real axis from $R$ to $r$, \\
$\; I(r):$ the circle of radius $r$ with center at the origin (oriented clockwise). \\
Then we have $\int_{I(R)} \to 0$ when $R \to \infty$ and $\int_{I(r)} \to 0$ when $r \to 0$. Hence it holds that
\begin{equation*}
\begin{split}
&\bigl(1-e^{2\pi is}\bigl)\int_0^\infty \!\! \frac{x^s dx}{2\pi i n(x-2\pi i n)}  = 
2 \pi i \frac{(2\pi i n)^s}{2\pi i n} = (2\pi)^s e^{\pi is/2} n^{s-1} ,\\
&\bigl(1-e^{2\pi is}\bigl)\int_0^\infty \!\! \frac{x^s dx}{2\pi i n(x+2\pi i n)}  = 
2 \pi i \frac{(-2\pi i n)^s}{2 \pi in} = (2\pi)^s e^{3\pi is/2} n^{s-1}
\end{split}
\end{equation*}
by the residue theorem. The formulas above imply Lemma \ref{lem:mt}.
\end{proof}

\begin{proof}[Proof of (\ref{eq:hueq1})]
For $0<a<1$ and  $-1 < \Re (s) < 0$, one has
\begin{equation*}
\begin{split}
\Gamma (s) &\zeta (s,a) = \int_0^\infty \!\! G(a,x) x^{s-1} dx = 
\int_0^\infty \sum_{n=1}^\infty \frac{e^{-2\pi i na}x^s dx}{2\pi i n(x \!-\! 2\pi i n)} -
\int_0^\infty \sum_{n=1}^\infty \frac{e^{2\pi i na} x^s dx}{2\pi i n(x \!+\! 2\pi i n)} \\ &= 
\frac{2i \pi i (2\pi)^{s-1} e^{\pi is/2}}{e^{2\pi is}-1} \sum_{n=1}^\infty \frac{e^{-2\pi ina}}{n^{1-s}} -
\frac{2i \pi i(2\pi)^{s-1} e^{3\pi is/2}}{e^{2\pi is}-1} \sum_{n=1}^\infty \frac{e^{2\pi ina}}{n^{1-s}}\\&=
\frac{\pi i (2\pi)^{s-1} e^{-\pi is/2}}{\sin \pi s} \sum_{n=1}^\infty \frac{e^{-2\pi ina}}{n^{1-s}} -
\frac{\pi i (2\pi)^{s-1} e^{\pi is/2}}{\sin \pi s} \sum_{n=1}^\infty \frac{e^{2\pi ina}}{n^{1-s}}.
\end{split}
\end{equation*}
by (\ref{eq:ft}), Proposition \ref{pro:irhur1} and Lemma \ref{lem:mt}. The interchange the order of the integral and the summation are justified by Lebesgue's dominated convergence theorem.
\end{proof}

\subsection{The case ${\bm{z\ne 1}}$}
Next, we prove the functional equation
\begin{equation}\label{eq:hueq2}
\Phi (s,a,z) = z^{-a} \Gamma (1-s) \sum_{n=\infty}^{-\infty} (-\log z+ 2\pi in )^{s-1} e^{2\pi ina}, 
\end{equation}
where $0<a<1$, $z\ne 1$ and $\Re (s) <0$ (see \cite[p.~28, (6)]{Er}) by the using integral representation (\ref{eq:hlzinrp2}). The following is a generalization of Lemma \ref{lem:mt}.
\begin{lemma}\label{lem:zne1}
For $0<a<1$ and $z\ne 1$, we have 
\begin{equation*}
\frac{e^{(1-a)x}}{e^x-z} = \lim_{N\to \infty}\sum_{n=N}^{-N} \frac{z^{-a}e^{-2 \pi ina}}{x-2\pi in -\log z}, \qquad
\frac{-1}{1-z} = \lim_{N\to \infty}\sum_{n=N}^{-N} \frac{z^{-a}e^{-2 \pi ina}}{2\pi in+\log z}.
\end{equation*}
\end{lemma}
\begin{proof}
The function $e^{(1-a)y}(e^y-z)^{-1}$ has poles at $y = 2\pi in + \log z$ with $n \in {\mathbb{Z}}$ and all of them are of order $1$. The residue at $y = 2\pi in + \log z$ is expressed as
$$
\lim_{y \to 2\pi in + \log z} (y-2\pi in - \log z) \frac{e^{(1-a)y}}{e^y-z} = \frac{e^{(1-a)(2\pi in + \log z)}}{z} = z^{-a}e^{-2 \pi ina}.
$$
Let $N$ be a sufficiently large natural number and $C_N$ be a square path passing through four corner points $R+iR$, $-R+iR$, $-R-iR$ and $R-iR$, where $R=2\pi (N+1/2)$, in this order. If $x$ is a point inside $C_N$ such that $x \ne 2\pi in + \log z$, one has
$$
\int_{C_N} \frac{e^{(1-a)y}}{e^y-z} \frac{dy}{y-x} = 2\pi i \biggl( \frac{e^{(1-a)x}}{e^x-z} - \sum_{n=N}^{-N} \frac{z^{-a}e^{-2 \pi ina}}{x-2\pi in -\log z} \biggr)
$$
from the residue theorem. When $N \to \infty$, the left hand side tends to $0$. Hence we have the first equation of this lemma (see also the proof of \cite[Theorem 4.11]{AIK}). We obtain the second formula of Lemma \ref{lem:zne1} by taking $x \to 0$. 
\end{proof}

From the lemma above, it holds that
\begin{equation}\label{eq:3.1n}
\begin{split}
&\frac{e^{(1-a)x}}{e^x-z} - \frac{1}{1-z} = \lim_{N\to \infty}\sum_{n=N}^{-N} 
\biggl( \frac{z^{-a}e^{-2 \pi ina}}{x-2\pi in-\log z} + \frac{z^{-a}e^{-2 \pi ina}}{2\pi in+\log z} \biggr) \\&=
\sum_{n=\infty}^{-\infty} \frac{x z^{-a}e^{-2 \pi ina}}{(2\pi in+\log z)(x-2\pi in-\log z)}.
\end{split}
\end{equation}
Furthermore, we have
\begin{equation}\label{eq:mhlt}
\int_0^\infty \frac{x^s dx}{x-2\pi in-\log z} = \frac{2 \pi i}{1-e^{2 \pi is}} (2\pi in+\log z)^s
\end{equation}
by modifying the proof of Lemma \ref{lem:mt}.

\begin{proof}[Proof of (\ref{eq:hueq2})]
When $0<a<1$, $z \ne 1$ and $-1 < \Re (s) <0$, we have
\begin{equation*}
\begin{split}
&\Phi (s,a,z) = \frac{1}{\Gamma (s)} \int_0^\infty \biggl( \frac{e^{(1-a)x}}{e^x-z} - \frac{1}{1-z} \biggr) 
x^{s-1} dx \\ &= \frac{1}{\Gamma (s)} \int_0^\infty 
\sum_{n=\infty}^{-\infty} \frac{x z^{-a}e^{-2 \pi ina} dx}{(2\pi in+\log z)(x-2\pi in-\log z)} = 
\frac{1}{\Gamma (s)} \frac{2 \pi i}{e^{2 \pi is}-1} \sum_{n=\infty}^{-\infty} 
\frac{-z^{-a}e^{-2 \pi ina}}{(2\pi in+\log z)^{1-s}} \\ &=
\frac{\Gamma (1-s)}{e^{ \pi is}} \sum_{n=\infty}^{-\infty} \frac{-z^{-a}e^{2 \pi ina}}{(-2\pi in+\log z)^{1-s}}
= z^{-a} \Gamma (1-s) \sum_{n=\infty}^{-\infty} \frac{e^{2 \pi ina}}{(2\pi in-\log z)^{1-s}}
\end{split}
\end{equation*}
from (\ref{eq:hlzinrp2}), (\ref{eq:3.1n}) and (\ref{eq:mhlt}). Thus we have (\ref{eq:hueq2}).
\end{proof}

 

\begin{thebibliography}{1}
\bibitem{AIK}
T.~Arakawa, T.~Ibukiyama and M.~Kaneko, {\it{Bernoulli numbers and zeta functions}}. {\it{With an appendix by Don Zagier}}. Springer Monographs in Mathematics. Springer, Tokyo, 2014. 

\bibitem{Apo} 
T.~M.~Apostol, \textit{Introduction to Analytic Number Theory}. Undergraduate Texts in Mathematics, Springer, New York, 1976.

\bibitem{Ber} 
B.~C.~Berndt, {\it{On the Hurwitz zeta-function}}. Rocky Mountain J.~Math. {\bf{2}} (1972), no.~1, 151--157. 

\bibitem{Er}
A. Erd\'elyi, W. Magnus , F. Oberhettinger and F. G. Tricomi, {\it{Higher transcendental functions Vol 1}}. McGraw-Hill, New York (1953).

\bibitem{GauLa}
R.~Garunk\v{s}tis and A.~Laurin\v{c}ikas, {\it{On zeros of the Lerch zeta-function}}. Number theory and its applications (Kyoto, 1997), 129--143, Dev.~Math. 2, Kluwer Acad.~Publ., Dordrecht, 1999. 

\bibitem{KKY}
S.~Kanemitsu, M.~Katsurada and M.~Yoshimoto, {\it{On the Hurwitz-Lerch zeta-function}}. Aequationes Math. {\bf{59}} (2000), no.~1-2, 1--19. 

\bibitem{KR}
M.~Knopp and S.~Robins, {\it{Easy proofs of Riemann's functional equation for $\zeta (s)$ and of Lipschitz summation}}. Proc.~Amer.~Math.~Soc. {\bf{129}} (2001), no.~7, 1915--1922. 

\bibitem{NaC}
T.~Nakamura, {\it{Real zeros of Hurwitz-Lerch zeta and Hurwitz-Lerch type of Euler-Zagier double zeta functions}}. to appear in Mathematical Proceedings of the Cambridge Philosophical Society. 

\bibitem{NRV}
L.~M.~Navas, F.~J.~Ruiz and J.~L.~Varona, {\it{The Lerch transcendent from the point of view of Fourier analysis}}. J.~Math.~Anal.~Appl. {\bf{431}} (2015),  no.~1, 186--201. 

\bibitem{OS}
C.~O'Sullivan {\it{Zeros of the dilogarithm}}. arXiv:1507.07980.

\bibitem{Pey}
A.~Peyerimhoff, {\it{On the zeros of power series}}. Michigan Math.~J. {\bf{13}} (1966) 193--214. 

\bibitem{Roy}
\'E.~L.~Roy. {\it{Sur les s\'eries divergentes et les fonctions d\'efinies par un d\'eveloppement de Taylor}}, Ann.~Fac.~Sci.~Toulouse Sci.~Math.~Sci.~Phys. (2), 2(3), (1900), 317--384.

\bibitem{Schi}
D.~Schipani, {\it{Zeros of the Hurwitz zeta function in the interval}} $(0,1)$, J.~Comb.~Number Theory {\bf{3}}  (2011),  no.~1, 71--74, (arXiv:1003.2060v3). 

\bibitem{Sie}
C.~L.~Siegel, {\it{Uber die Klassenzahl quadratischer Zahlkorper}}, Acta Arith. {\bf{1}} (1936), 83--86.

\bibitem{Spira}
R.~Spira, {\it{Zeros of Hurwitz zeta functions}}, Math.~Comp. {\bf{30}} (1976), no.~136, 863--866. 

\bibitem{Tit} E.~C.~Titchmarsh, {\it{The theory of the Riemann zeta-function,}} Second edition. Edited and with a preface by D.~R.~Heath-Brown. The Clarendon Press, Oxford University Press, New York, 1986. 

\bibitem{VW}
A.~P.~Veselov and J.~P.~Ward, {\it{On the real zeroes of the Hurwitz zeta-function and Bernoulli polynomials}}.
J.~Math.~Anal.~Appl. {\bf{305}} (2005), no.~2, 712--721. 

\end{thebibliography}
\end{document}